\documentclass[11pt,english,a4paper]{smfart}
\usepackage[english]{babel}
\usepackage{amssymb,xspace}
\usepackage{amstext}
\usepackage{smfthm}
\theoremstyle{plain}
\usepackage{amsbsy,amssymb,amsfonts,latexsym}

\usepackage{enumerate}
\usepackage{graphics}

\title[Recurrence and non-recurrence sets for products of rotations]{Some new examples of recurrence and non-recurrence sets for products of rotations on the unit circle}

\author{Sophie Grivaux}
\address{CNRS,
Laboratoire Paul Painlev\' e, UMR 8524, Universit\'e  Lille 1, Cit\' e Scientifique, 59655 Villeneuve d'Ascq
Cedex, France}
\email{grivaux@math.univ-lille1.fr}

\author{Maria Roginskaya}
\address{Department of Mathematical Sciences,
Chalmers University of Technology, SE-41296 G\"oteborg, Sweden, \emph{and}
Department of Mathematical Sciences,
G\"oteborg University, SE-41296 G\"oteborg, Sweden}
\email{maria@chalmers.se}



\thanks{The first author was partially supported
by ANR-Projet Blanc DYNOP}

\def\T{\ensuremath{\mathbb T}}
\def\R{\ensuremath{\mathbb R}}
\def\Z{\ensuremath{\mathbb Z}}

\def\Q{\ensuremath{\mathbb Q}}
\def\N{\ensuremath{\mathbb N}}

\newcommand{\wrt}{with respect to}

\newcommand{\mpt}{measure-preserving transformation}

\newtheorem{theorem}{Theorem}[section]

\newtheorem{lemma}[theorem]{Lemma}

\newtheorem{proposition}[theorem]{Proposition}

\newtheorem{corollary}[theorem]{Corollary}

{\theoremstyle{definition}}

{\theoremstyle{definition}}

{\theoremstyle{definition}}

{\theoremstyle{definition}}

{\theoremstyle{definition}\newtheorem{remark}[theorem]{Remark}}

\newtheorem{question}[theorem]{Question}

{\theoremstyle{definition}\newtheorem*{FFC Criterion}{Frequent
Faber-hypercyclicity Criterion}}
\newtheorem*{Hypercyclicity Criterion}{Hypercyclicity Criterion}
{\theoremstyle{definition}\newtheorem*{GS Criterion}{Godefroy-Shapiro
Criterion}}
\def\piednote#1{\let\oldfn=\thefootnote\def\thefootnote{}\footnote{\noindent#1}%
\addtocounter{footnote}{-1}\def\thefootnote{\oldfn}}

\begin{document}

\begin{abstract}
We study recurrence and non-recurrence sets for dynamical systems on compact spaces, in particular for products of rotations on the unit circle $\T$. A set of integers is called $r$\emph{-Bohr} if it is recurrent for all products of $r$ rotations on $\T$, and \emph{Bohr} if it is recurrent for all products of rotations on $\T$. It is a result due to Katznelson that for each $r\ge 1$ there exist sets of integers which are $r$-Bohr but not $(r+1)$-Bohr. We present new examples of $r$-Bohr sets which are not Bohr, thanks to a construction which is both flexible and completely explicit. Our results are related to an old combinatorial problem of Veech concerning syndetic sets and the Bohr topology on $\Z$, and its reformulation in terms of recurrence sets which is due to Glasner and Weiss. 
\end{abstract}
\maketitle

\section{Introduction}

The general topic of this paper is the study of recurrence and non-recurrence sets for dynamical systems. In the topological setting, recurrence sets are defined as follows: a dynamical system is a pair $(X,d,f)$, where $(X,d)$ is a compact metric space for the distance $d$ and $f$ is a continuous map of $(X,d)$ into itself. If $(n_k)_{k\ge 0}$ is a strictly increasing sequence, we say that $\{n_k\}$  is  a \emph{recurrence set} (or a \emph{Birkhoff set}) if for any  dynamical system $(X,f)$ and any $\epsilon>0$ there exists a point $x\in X$ and a $k\geq 0$ such that $d(f^{n_k}(x),x)< \epsilon$, where $f^{n}={f\circ\ldots\circ f}$ ($n$ times) denotes the $n^{th}$ iterate of $f$. In the measure theoretic
setting recurrence sets, which are often called \emph{Poincar\'e sets} in this context, are defined in this way: $\{n_k\}$ is a recurrence set if for any probability space $(X,\mathcal{B},m)$ and any measure-preserving transformation $T$ of $X$, there exists for any set $A\in\mathcal{B}$ with $m(A)>0$ and a $k\geq 0$ such that $m(T^{n_k}A\cap A)>0$.

\par\smallskip

It is not difficult to see that any Poincar\'e set is a Birkhoff set: indeed if $\{n_k\}$ is not a Birkhoff set, let $(X,d,f)$ be a dynamical system such that for some $\epsilon >0$, $d(f^{n_k}(x),x)\geq \epsilon$ for any $k\geq 0$ and any $x\in X$. Without loss of generality, $(X,d,f)$ can be supposed to be a minimal system, and hence there exists a probability measure $m$ on $X$ whose support is $X$ and which is invariant by $f$. 
There exists then a non-empty open set $U$ in $X$ such that $f^{n_k}(U)\cap U=\varnothing$ for any $k\geq 0$. As $m(U)>0$, it follows that $\{n_k\}$ is not a Poincar\'e set. The converse assertion is not true: there are Birkhoff sets which are not Poincar\'e sets \cite{Kr}, see also \cite{Weiss}.

\par\smallskip

Recurrence is a central topic in the study of dynamical systems, and we refer the reader to one of the classical books \cite{Wa} or \cite{Pe} for the basic facts, and to the works \cite{Fu2}, \cite{Fu}, \cite{G} or \cite{GW} for a deeper study of various recurrence properties, as well as their applications to number theory and combinatorics.

\par\smallskip

In the rest of the paper, we say that $\{n_k\}$ is \emph{a recurrence set for the dynamical system} $(X,d,f)$ if for all $\epsilon >0$ there exists $k\geq 0$ and $x\in X$ such that $d(f^{n_k}(x),x)<\epsilon$, and that it is a \emph{recurrence set in the ergodic sense} for $(X,m,T)$ if for any $A\in \mathcal{B}$ with $m(A)>0$ there exists a $k\geq 0$ such that $m(T^{n_k}A\cap A)>0$. 

\par\smallskip

Standard examples of recurrence sets (besides the obvious example of the set $\{k\}$) are the set of squares $\{k^2\}$, 
or more generally the sets of the form $\{p(k)\}$ where $p$ is a polynomial taking integer values on integers with $p(0)=0$,
difference sets $D-D$ where $D$ is any infinite set in $\mathbb{N}$, thick sets (i.e. sets containing arbitrarily long intervals), the sets $\mathcal{P}-1$ and $\mathcal{P}+1$, where $\mathcal{P}$ denotes
 the set of primes, or more generally the so called van der Corput (vdC) sets. See, for instance \cite[p. 109]{Queff} or \cite{BL} for more information on vdC sets.
Some generalized polynomials also yield recurrence sets, see \cite{BH}.

\par\smallskip

The starting point of this paper is  an old problem in combinatorial number theory which is to know whether any difference set $S-S$, where $S$ is a subset of $\mathbb{Z}$ with bounded gaps, must contain a Bohr neighborhood of zero. It is known by a result of Veech \cite{Ve} that this is true up to a set of density zero, but it is not known whether this set can be dispensed with.  It is shown by
Glasner in \cite{Gl2} and Boshernitzan and Glasner in \cite{BoGl} (see also the papers \cite{GW} by Glasner and Weiss and \cite{Weiss} by Weiss) that this problem is equivalent to the following question concerning recurrence sets:

%
%

\begin{question} \cite{Gl2}, \cite{GW}, \cite{BoGl}, \cite{Weiss} \label{q2}
If $\{n_k\}$ is a recurrence set for all finite products of circle rotations, is it a recurrence set? 
\end{question}
 
Question \ref{q2} was studied in several papers, for instance in \cite{Gl2}, \cite{Weiss}, \cite{BoGl}, \cite{Pest} and \cite{K} (where an equivalent formulation in terms of Cayley numbers of graphs is given). Sets which are recurrent for all finite products of circle rotations are called \emph{Bohr sets}. If $r$ is a positive integer, a set which is recurrent for all products of $r$ rotations on $\T^{r}$ is called \emph{$r$-Bohr}. In view of Question \ref{q2}, it comes as a natural problem to ask whether an $r$-Bohr set is necessarily a Bohr set. It was shown by Katznelson in \cite{K} that it is not the case. More precisely, the following result was proved in \cite{K}:

\begin{theorem}\label{thkatz}\cite{K}
 Let $r\ge 1$. For each $(r+1)$-tuple $(\lambda _{1},\ldots,\lambda _{r+1})$ of elements of $\T$, $\lambda _{j}=e^{2i\pi\theta _{j}}$, $\theta _{j}\in[0,1)$ with $(\theta _{1},\ldots, \theta  _{r+1})$ $\Q$-independent, and for each $\delta \in (0,1)$, the set
 $$D^{\delta }_{\lambda _{1},\ldots, \lambda _{r+1}}=\{n\ge 0 \textrm{ ; }\min_{j=1,\ldots, r+1}|\lambda _{j}^{n}+1|<\delta \}$$ is an $r$-Bohr set which is not $(r+1)$-Bohr.
\end{theorem}

The sets $D^{\delta }_{\lambda _{1},\ldots, \lambda _{r+1}}$ are ``large'' sets in the sense that they have positive density. It is possible to obtain from Theorem \ref{thkatz} many Bohr sets: given a sequence $(\delta _{r})_{r\ge 1}$ of numbers in $(0,1)$ and families $((\lambda _{1,r},\ldots, \lambda _{r+1,r}))_{r\ge 1}$, the set
$D=\bigcup_{r\ge 1}D^{\delta_{r} }_{\lambda _{1,r},\ldots, \lambda _{r+1,r}}$ is obviously a Bohr set. However, it is clear that the set $D$ is a Poincar\'e set: for each $r\ge 1$, the $(r+1)$-tuple $(\lambda _{1,r}, \lambda _{2,r}^{2}\ldots, \lambda _{r+1,r}^{r+1})$ is $\Q$-independent in the sense of Theorem \ref{thkatz}, and thus the set
$D^{\delta_{r} }_{\lambda _{1,r},\ldots, \lambda _{r+1,r}}$ contains, for some $q\ge 1$, the integers $q,2q,\ldots, (r+1)q$. This implies that for any \mpt\ $T$ of a probability space $(X, \mathcal{B},m)$, any set $A\in \mathcal{B}$ which is such that $m(T^{n}A\cap A)=0$ for each $n\in D^{\delta_{r} }_{\lambda _{1,r},\ldots, \lambda _{r+1,r}}$ is such that $m(A)<\frac{1}{r}$. The appearance of such sequences $(q,2q,\ldots, (r+1)q)$ comes from the particular structure of the sets $D^{\delta_{r} }_{\lambda _{1,r},\ldots, \lambda _{r+1,r}}$, and it is natural to wonder whether it is possible to construct other kinds of $r$-Bohr sets which are not Bohr, which would have a different arithmetical structure and come closer to a potential counterexample to Question \ref{q2}. It is the aim of this paper to provide an alternative construction of $r$-Bohr sets which are not Bohr, which has the advantage over the construction of \cite{K} to be both very flexible and completely explicit. Our main result can be stated as follows:

%
%

\begin{theorem}\label{th2} 
For each $r\geq 1$ there exist sets $\{n_{k}^{(r)}\}$ of integers which are $r$-Bohr but not $(2^{r-1}+1)$-Bohr, and which have the following simple structure:

$$\{n_{k}^{(r)}\}=\{n_{k,0}^{(r)}\}\cup\bigcup_{A\subseteq \{1,\ldots, r-1\}}\{n_{k,A}^{(r)}\}$$
where
$$
\{n_{k,0}^{(r)}\}=\bigcup_{N\ge 1}B_{N,0}^{(r)}
\quad \textrm{ and }\quad 
\{n_{k,A}^{(r)}\}=\bigcup_{N\ge 1}B_{N,A}^{(r)}, \quad  A\subseteq\{1,\ldots, r-1\}
$$
and
\begin{eqnarray*}
B_{N,0}^{(r)}&=& B_{\varepsilon _{N}^{(r)},0}^{(r)}=\{H_{N}q+ 1 \textrm{ ; } 1\le q\le Q_{N}^{(r)}\}\\
B_{N,\varnothing}^{(r)}&=& B_{\varepsilon _{N}^{(r)},\varnothing}^{(r)}=\{H_N  \Delta_{N,\varnothing}^{(r)}\}\\
B_{N,A}^{(r)}&=& B_{\varepsilon _{N}^{(r)},A}^{(r)}= \{H_N\Delta    _{N,A}^{(r)} (L_N j + 1)\textrm{ ; } 1\le j\le \Theta_N^{(r)} \},
\end{eqnarray*}
where $(L_{N})_{n\ge 1}$ is a rapidly growing sequence of integers, $(\Delta    _{N,A}^{(r)})_{N\ge 1}$, $(\Theta_N^{(r)})_{N\ge 1}$ and $(Q_{N}^{(r)})_{N\ge 1}$ are sequences of integers depending from $(L_{N})_{N\ge 1}$, and $(H_{N})_{N\ge 1}$ is a very rapidly increasing sequence of integers independent from all the other parameters.
\end{theorem}

The arithmetic structure of these sets $\{n_{k}^{(r)}\}$ is very explicit, and one can construct from them
many examples of Bohr sets. But contrary to the sets from \cite{K}, it is for most choices of parameters in the construction not clear whether these sets are recurrent sets or not. So our construction does not solve Question \ref{q2}, but highlights how delicate this question is.
\par\smallskip
The paper is organized as follows: Section $2$ is devoted to the proof of Theorem \ref{th2} in the case where $r=1$ (here it is completely elementary). The proof of Theorem \ref{th2} for general $r$ is the object of Sections $3$ and $4$.
 Lastly, we construct in Section $6$ some Bohr sets obtained from Theorem \ref{th2}, and present some final comments and remarks.
\par\smallskip
In the whole paper we will denote by $R_{\lambda}$ the rotation on $\T$ associated to $\lambda\in\T$.

\section{Proof of Theorem \ref{th2} for $r=1$}

Let us begin by recalling what we want to prove: we are looking for a set $\{n_{k}\}$ of the form given in Theorem \ref{th2} which is recurrent for all circle rotations, i.e. such that
\begin{center}
 for any $\lambda \in\T$, any $\varepsilon >0$, there exists a $k$ such that $|\lambda ^{n_{k}}-1|<\varepsilon $
\end{center}
but which is not recurrent for all products of two circle rotations, i.e. for which there exist $\mu _{0},\mu _{1}\in\T$ and $\delta >0$ such that
\begin{center}
for any $k\ge 0$, $\max(|\mu _{0}^{n_{k}}-1|, |\mu_{1}^{n_{k}}-1|)>\delta $. 
\end{center}
We will use the following notation
$$M_{1}=\inf_{\{\theta \}\not =0}\dfrac{|e^{2i\pi \theta }-1|}{\{\theta \}}\quad \textrm{ and }\quad 
M_{2}=\sup_{\{\theta \}\not =0}\dfrac{|e^{2i\pi \theta }-1|}{\{\theta \}}\cdot$$
We will denote by $\lfloor\theta \rfloor$ the integer part of the real number $\theta $, and by $\{\theta \}$ its distance to $\Z$. We will also need the following simple fact:

\begin{lemma}\label{lem1bis}
There exist two universal constants $C,C'\ge 1$ such that for any $\gamma>0$ and $\varepsilon>0$, for any $\mu \in\T$, the following holds true:

if $\gamma<|\mu-1|<\varepsilon$, then for any $\nu\in\T$ there exists an integer $p$ with  $1\le p\le \lfloor\frac{C'}{\gamma}\rfloor$ such that
$|\mu^p-\nu|\le C \varepsilon$.
\end{lemma}

\begin{proof}[Proof of Lemma \ref{lem1bis}]
Write $\mu $ as $\mu =e^{2i\pi\theta }$, $|\theta |\le\frac{1}{2}$. Without loss of generality, we can suppose that $\theta >0$. We have $M_{1}\theta \le |e^{2i\pi \theta }-1|\le M_{2}\theta $, and thus
$\frac{\gamma }{M_{2}}<\theta <\frac{\varepsilon }{M_{1}}$. Let $\kappa =\lfloor \frac{M_{2}}{\gamma }\rfloor$. Since $\{\theta \}>\frac{1}{\kappa }$, the fractional parts of the $\kappa $ numbers $\theta , 2\theta , \ldots, \kappa \theta $ form a $\theta $-net of $(0,1)$: for any $\alpha \in\R$
 there exists a $p$ with $1\le p\le \kappa $ such that $\{p\theta -\alpha \}\le \theta <\frac{\varepsilon }{M_{1}}$. Hence
 $$|e^{2i\pi p\theta }-e^{2i\pi\alpha }|\le M_{2}\frac{\varepsilon }{M_{1}},$$ and this proves Lemma \ref{lem1bis} with $C=\frac{M_{2}}{M_{1}}$ and $C'=M_{2}$.
\end{proof}

The key lemma for the proof of Theorem \ref{th2} in the $1$-dimensional case is the following:

\begin{lemma}\label{lem1}
For any $\varepsilon >0$ there exist two positive integers $\Sigma, \Theta \ge 1  $ such that for any $\lambda \in\T$,
any integers $L\ge 1$ and
$H\ge 1$, and any $S\in\Z$, one of the following two assertions is true: 
either
\begin{eqnarray}\label{eq2}
 |\lambda ^{H \Sigma L  }-1|<\varepsilon 
\end{eqnarray}

or
\begin{eqnarray}\label{eq1}
\textrm{there exists a } j\in \{1,\ldots, \Theta  \} \textrm{ such that }|\lambda ^{HLj+S}-1|<\varepsilon .
\end{eqnarray}
\end{lemma}

We will apply Lemma \ref{lem1} with two values of $S$ only: $S=1$ and $S=H$. In the first case the set of integers appearing in (\ref{eq1}) is simply a shifted arithmetic progression of step $HL$, and in the second case the set we get a multiple of a shifted arithmetic progression of step $L$.

\begin{proof}[Proof of Lemma \ref{lem1}]
The idea of the proof can be summarized as follows: define an integer $\kappa$ as $\kappa =\lfloor\frac{4\pi C C'}{\varepsilon }\rfloor$. If 
$\lambda ^{HLl}$ is not too close to $1$ for some $l\in\{1,\ldots, \kappa \}$, then by Lemma \ref{lem1bis}
any $\mu \in\T$ (in particular $\lambda^{-S} $) can be $\varepsilon $-approximated by a power of $\lambda^{HL} $ which is not too large, and (\ref{eq1}) is true. If $\lambda ^{HLl}$ is too close to $1$ for each $l\in\{1,\ldots, \kappa \}$, then (\ref{eq2}) holds true. Let us now be more precise, and consider the quantity $\gamma=\min_{l=1,\ldots,\kappa }|\lambda^{HLl}-1|$. By the Dirichlet principle, we know that $\gamma \le \frac{M_{2}}{\kappa }< \frac{\varepsilon }{C}$ since $M_{2}=C'$. There are two cases to consider.
\par\smallskip
\textbf{Case 1:} we have $\gamma < \frac{\varepsilon}{4\pi C}\frac{1}{\kappa !}\cdot$
\par\smallskip
This means that there exists an $l\in \{1,\ldots, \kappa \}$ such that $|\lambda^{HLl}-1|<
\frac{\varepsilon}{4\pi C}\frac{1}{\kappa !}$. Since $1\le l \le \kappa $, $l$ divides $\kappa !$, and so it makes sense to write
\begin{eqnarray*}
|(\lambda ^{HLl})^{\frac{\kappa !}{l}}-1|=|\lambda^{H L\kappa !}-1| <\frac{\kappa !}{l}\frac{\varepsilon}{4\pi C}\frac{1}{\kappa !} \le \frac{\varepsilon}{4\pi C} \cdot
\end{eqnarray*}
So (\ref{eq2}) is true with $\Sigma  =\kappa !$.
\par\smallskip
\textbf{Case 2:} we have $\gamma \ge \frac{\varepsilon}{4\pi C}\frac{1}{\kappa !}\cdot$
\par\smallskip
This implies that there exists an $l\in \{1,\ldots, \kappa \}$ such that
\begin{eqnarray*}
\frac{\varepsilon}{4\pi C}\frac{1}{\kappa !}\le|\lambda^{HL l}-1|\le\gamma < \frac{\varepsilon}{ C} .
\end{eqnarray*}
By Lemma \ref{lem1bis}, there exists an integer $p\in\{1,\ldots, \lfloor \frac{4\pi C C'}{\varepsilon}\kappa !\rfloor\}$
 such that
\begin{eqnarray*}
|\lambda^{HLlp}-\lambda^{-S}|=|\lambda^{HLlp+S}-1| < \varepsilon .
\end{eqnarray*}
Since $1\le lp\le \kappa \lfloor \frac{4\pi C C'}{\varepsilon}\kappa !\rfloor$ we get, setting $\Theta =\kappa \lfloor \frac{4\pi C C'}{\varepsilon}\kappa !\rfloor$, a $j\in\{1,\ldots, \Theta \}$ such that
$$|\lambda ^{HLj+S}-1|< \varepsilon  $$
and (\ref{eq1}) is  true. Lemma \ref{lem1} is proved.
\end{proof}

\begin{remark}\label{rem00}
Let us record here for further use the expression of $\Sigma$ and $\Theta$ which we obtained in the proof of Lemma \ref{lem1}:
$$\Sigma= (\lfloor\frac{4\pi C C'}{\varepsilon }\rfloor)!\quad \textrm{and} \quad \Theta =\lfloor\frac{4\pi C C'}{\varepsilon }\rfloor
\lfloor \frac{4\pi C C'}{\varepsilon}(\lfloor\frac{4\pi C C'}{\varepsilon }\rfloor)!\rfloor.$$ Observe that $\Theta$ is much larger than $\Sigma$ and that given any integer $A\ge 1$, one can ensure by taking $\varepsilon$ sufficiently small that $A$ divides $\Sigma$.
\end{remark}

We are now ready for the proof of Theorem \ref{th2} for $r=1$.

\begin{proof}[Proof of Theorem \ref{th2} for $r=1$]
We construct two sets $\{n_{k,0}\}$ and $\{n_{k,1}\}$ by induction on $N$ and by blocks, applying repeatedly Lemma \ref{lem1}. We start by taking at the first step $\varepsilon _{1}=2^{-1}$. Lemma \ref{lem1} gives us a $\Theta _1$ and a $\Sigma _{1}$, then we choose $L_{1}=1$, $S_{1}=1$, a large even number $H_1$, and we take for the first $\Theta _1$ elements of the set
$\{n_{k,0}\}$ the numbers
$$H_1+1, 2H_1+1,\ldots,  H_1 \Theta _1 +1$$
which are all odd. For the first element of the set $\{n_{k,1}\}$ we take the number $H_1 \Sigma _1$, which is even. Then we take $\varepsilon _2=2^{-2}$, obtain $\Theta _{2}$ and $\Sigma _{2}$,  then take $L_{2}=1$, $S_{2}=1$ and $H_2$ very large and even (much larger than $ H_1 \Theta _1$ in particular). We continue the set $\{n_{k,0}\}$ with the numbers
$$H_2+1, 2H_2+1,\ldots,  H_2 \Theta _2 +1,$$ and we take for the second element of the set $\{n_{k,1}\}$ the number $H_2 \Sigma  _2$. We continue in this fashion:
$$\{n_{k,0}\}=\bigcup_{p\ge 1}\{H_{p}+1, 2H_{p}+1,\ldots, H_{p} \Theta  _{p}+1\}$$
and
$$\{n_{k,1}\}=\{H_{p}\Sigma  _{p} \textrm{ ; } p\ge 1\}$$
where $\Sigma  _{p}$ and $\Theta  _{p}$ result from the application of Lemma \ref{lem1} to $\varepsilon _{p}=2^{-p}$, where
 $L_{p}=1$, $S_{p}=1$, and the sequence $(H_{p})$ consists of even numbers and increases very rapidly. Observe that these two sets are disjoint, since all the elements in the first set are odd while all elements in the second set are even, and that $\Sigma  _{p}$ is much smaller than $\Theta _{p}$ by Remark \ref{rem00}, so that the set $\{n_{k}\}=\{n_{k,0}\}\cup
\{n_{k,1}\}$ looks like this:
$$\{n_{k}\}=\bigcup_{p\ge 1}\{H_{p}+1, 2H_{p}+1,\ldots, H_{p}(\Sigma  _{p}-1)+1, 
H_{p}\Sigma _{p}, H_{p}\Sigma  _{p}+1, \ldots, H_{p}\Theta  _{p}+1\}.$$
Now by Lemma \ref{lem1}, it is clear that for all $\lambda \in\T$ and all $p\ge 1$ there exists a $k$ such that
$|\lambda ^{n_{k}}-1|<2^{-p}$, so $\{n_{k}\}$ is a recurrence set for any rotation of $\T$.

It remains to find $\mu _{0},\mu _{1}\in\T$ such that $\{n_{k}\}$ is not recurrent for $R_{\mu _{0}}\times R_{\mu _{1}}$, and this is not difficult thanks to the particular structure of the set $\{n_{k}\}$. We will need the following lemma, which is implicit in \cite{BaGr}:

\begin{lemma}\label{lem2}
There exists a positive constant $M$ such that if 
 $(m_{k})_{k\ge 1}$ is any sequence of integers such that $\frac{m_{k+1}}{m_{k}}>2$ for all $k \ge 1$, there exist uncountably many $\lambda \in\T$ such that for all $k\ge 1$,
$$|\lambda ^{m_{k}}-1|\le  M  \,\frac{m_{k}}{m_{k+1}}\cdot$$
Moreover, the set of such $\lambda$'s is $\frac{6\pi}{m_1}$-dense in $\T$.
 In particular if $\frac{m_{k+1}}{m_{k}}\to +\infty $, there exists an element $\lambda \in\T$ with $|\lambda+1|\le \frac{6\pi}{m_1}$ such that 
$\lambda ^{m_{k}}$ tends to $1$ at the rate $\frac{m_{k}}{m_{k+1}}$.
\end{lemma}

For completeness's sake we provide a short proof of Lemma \ref{lem2}.

\begin{proof}[Proof of Lemma \ref{lem2}]
For any $k\ge 1$, let $E_{k}$ be the set
$$E_{k}=
\left\{\lambda =e^{2i\pi\theta }\in \T\textrm{ ; } \{m_{k}\theta \} \in \left[0, 2 \frac{m_{k}}{m_{k+1}}\right]\right\}.$$ The set $E_{k}$ is the union of a collection of disjoint closed sub-arcs of $\T$ of length 
$\frac{8\pi}{m_{k+1}}$. We write this collection as $\{I_{j}^{(k)}\}$.  The distance between two consecutive such sub-arcs is equal to
$\frac{2\pi}{m_{k}}-\frac{8\pi}{m_{k+1}}$.
So any arc $I$ of $\T$ of length greater than $\frac{4\pi}{m_{k}}+\frac{8\pi}{m_{k+1}}$ contains two arcs of the collection $\{I_{j}^{(k)}\}$. Now observe that
$\frac{8\pi}{m_{k+1}}>\frac{4\pi}{m_{k+1}}+\frac{8\pi}{m_{k+2}}$, because
$m_{k+2}>2m_{k+1}$.
It follows that any arc $I_{j}^{(k)}$ contains two disjoint arcs of the collection $\{I^{(k+1)}_{j'}\}$, and in this way we construct a Cantor-type subset $K$ of $(0,1)$ such that for all $k$ and all $\theta \in K$, $\{m_{k}\theta\} \le 2 \frac{m_{k}}{m_{k+1}}$. So any $\lambda =e^{2i\pi\theta }$ with $\theta \in K$ satisfies 
$$|\lambda ^{m_{k}}-1|\le 2M_{2} \,\frac{m_{k}}{m_{k+1}}\quad \textrm{for any } k\ge 1.$$ 
Since any subarc of the set $E_1$ contains a $\lambda=e^{2i\pi\theta}$ with $\theta\in K$, the set of such $\lambda$'s is $(\frac{2\pi}{m_1}+\frac{8\pi}{m_2})$-dense in $\T$, so it is $\frac{6\pi}{m_1}$-dense in $\T$.
 Lemma \ref{lem2} is proved with $M=2M_{2}$.
\end{proof}
Let us now go back to the proof. The crucial observation is that the sequence $(H_{p})_{p\ge 1}$ may be chosen as rapidly growing as we want to. Let $\mu_{0}\in\T$,  given by Lemma \ref{lem2}, be such that  $|\mu _{0}+1|\le\frac{6\pi}{H_1}$
and
for all $p\ge 1$, $$|\mu _{0}^{H_{p}}-1|\le  M \,\frac{H_{p}}{H_{p+1}}\cdot $$ Then for any $j=1,\ldots, \Theta _{p}$ we have 
$$|\mu _{0}^{H_{p}j}-1|\le  M \Theta  _{p}\,\frac{H_{p}}{H_{p+1}}\cdot $$ If $H_1\ge 6\pi$ and $H_{p+1}$ is sufficiently large \wrt\ $\Theta _{p}$ and $H_{p}$, we can ensure that 
$|\mu _{0}^{H_{p}j}-1|\le  2^{-p} $ for all $p\ge 1$ and $j=1,\ldots, \Theta _{p}$, i.e. that 
$|\mu _{0}^{H_{p}j+1}-{\mu}_{0}|\le  2^{-p} $. Now
$|\mu _{0}^{H_{p}j+1}-1|\ge |\mu_0-1|-2^{-p}\ge 2-\frac{6\pi}{H_1}-2^{-p}\ge \frac{1}{2}$.
Hence we get that
$|\mu _{0}^{n_{k,0}}-1|\ge \frac{1}{2}$ for all $k$. The argument for the construction of $\mu _{1}$ is exactly similar: Lemma \ref{lem2} gives us a $\mu _{1}$ 
such that $|\mu_1+1|\le\frac{6\pi}{H_1}$ and
$$|\mu _{1}^{H_{p}\Sigma  _{p}-1}-1|<2^{-p}$$ for all $p\ge 1$, and so $|\mu _{1}^{H_{p}\Sigma  _{p}}-\mu _{1}|<2^{-p}$. Hence $|\mu _{1}^{n_{k,1}}-1|\ge \frac{1}{2}$ for all $k$. Putting things together we get that for all $k$, 
$$\max(|\mu _{0}^{n_{k}}-1|,|\mu _{1}^{n_{k}}-1|)\ge \frac{1}{2}, $$ which is exactly what we wanted to prove.
An easy modification of the proof shows that we can replace the bound $\frac{1}{2}$ above by any $\delta \in (0,2)$ as close as we want to $2$ (it suffices to take $H_{1}$ extremely large and to replace the quantities $2^{-p}$ in the estimates above by $a^{-p}$, where $a$ is some suitably large integer): for any $\delta \in (0,2)$ there exists a set $\{n_{k}\}$ which is recurrent for all rotations, but such that there exist $\mu _{0},\mu _{1}\in\T$ such that  $\max(|\mu _{0}^{n_{k}}-1|,|\mu _{1}^{n_{1}}-1|)>\delta $ for all $k$.
\end{proof}

\begin{remark}\label{rem1}
 We did not use the parameter $L$ nor the parameter $S$ of Lemma \ref{lem1} in this construction, but they will be necessary later on in the proof of Theorem \ref{th2}.  Lemma \ref{lem1} also gives us other examples of sequences which are recurrent for all rotations, but not recurrent for some product of two rotations. For instance the proof would work as well if we considered for $\{n_{k,0}\}$ and $\{n_{k,1}\}$ the sets $$\{n_{k,0}\}=\bigcup_{p\ge 1}\{H_{p}, 3H_{p},\ldots, H_{p}(2\Theta  _{p}+1)\} \;\textrm{ and }\;\{n_{k,1}\}=\{2H_{p}\Sigma  _{p} \textrm{ ; } p\ge 1\}$$
 ($S_{p}=H_p$ and $L_{p}=2$) or 
 $$\{n_{k,0}\}=\bigcup_{p\ge 1}\{H_{p}(L_{p}j+1) \textrm{ ; } 1\le j\le \Theta_{p}
\} \;\textrm{ and }\;\{n_{k,1}\}=\{H_{p}\Sigma  _{p}L_{p} \textrm{ ; } p\ge 1\}$$ where $(L_{p})$ is a rapidly increasing sequence.
\end{remark}

The proof of Theorem \ref{th2} uses induction on $r\ge 2$ and the same kind of ideas, but becomes more involved as the dimension grows. In order to make the underlying ideas of the induction clear, we will present the $2$-dimensional case first.

\section{The $2$-dimensional case of Theorem \ref{th2}}
%
%

The first difficulty one encounters when trying to go from the $1$-dimensional to the $2$-dimensional case is that one needs a multi-dimensional analogue of the following fact, which is the crux of the proof of Lemma \ref{lem1}: if $\theta\in (0,1)$, then the numbers $$e^{2i\pi\theta},(e^{2i\pi\theta})^{2},\ldots, (e^{2i\pi\theta})^{p}$$ where $p=\lfloor\frac{1}{\theta}\rfloor$ form a $2\pi\theta$-net of the unit circle, and an important point is the dependence of $p$ from $\theta$. In the multi-dimensional case, we will use a weak form of the following result of Kannan and Lovasz \cite{KL}:

\begin{theorem}\label{th0ter}
Let $(\alpha_1,\ldots, \alpha_r)$ be an $r$-tuple of real numbers, and $\varepsilon>0$. Suppose that $Q$ is an integer such that for any $r$-tuple 
$(a_1,\ldots, a_r)$ of elements of $\Z$, which are not all zero, the following inequality holds
\begin{eqnarray}\label{eq3}
Q\left\{\sum_{i=1}^r a_i\alpha_i\right\}+\varepsilon \sum_{i=1}^r |a_i| \ge c_0 r^2
\end{eqnarray}
where $c_0$ is some positive universal constant. Then for all $(\beta_1,\ldots, \beta_r)\in\R^{r}$ there exist $(p_1,\ldots, p_r)\in\Z^{r}$ and $q\in \Z$ with $|q|\le Q$ such that
\begin{eqnarray*}
|q\alpha_i-p_i-\beta_i|\le \varepsilon \quad \textrm{ for each } i=1,\ldots, r.
\end{eqnarray*}
\end{theorem}

Condition (\ref{eq3}) quantifies ``how independent'' the reals $(\alpha_1,\ldots, \alpha_r)$ are, and the size of the bound $Q$ depends from $\varepsilon $ and from this rate of independence. Recall that $\{x\}$ denotes the distance of the real number $x$ to $\Z$.
\par\smallskip
Theorem \ref{th0ter} has the following consequence (we disregard the particular expression of the bound $c_0 r^2$, which is actually an important issue in \cite{KL}):

\begin{corollary}\label{cor0} For each $r\ge 1$ there exists a positive constant $c_r$ such that the follo\-wing statement holds true for any
 $(\lambda_1,\ldots, \lambda_r)\in\T^{r}$:
if $\varepsilon>0$ and $Q\ge 1$, with $Q$ an integer, are such that for any 
$(a_1,\ldots, a_r)\in\Z^{r}\setminus\{(0,\ldots,0)\}$
\begin{eqnarray}\label{eq4}
Q|\lambda_1^{a_1}\lambda_2^{a_2}\ldots\lambda_r^{a_r}-1|+\varepsilon \sum_{i=1}^r |a_i| \ge c_r,
\end{eqnarray}
then for any $(\mu_1,\ldots, \mu_r)\in\T^{r}$ there exists  a $q\in \N$ with $1\le q\le Q$ such that
\begin{eqnarray*}
|\lambda_i^q-\mu_i|<\varepsilon \quad \textrm{ for each } i=1,\ldots, r.
\end{eqnarray*}
\end{corollary}
This is the multi-dimensional extension of Lemma \ref{lem1bis} which will be needed in the rest of the proof. Let us now go back to the $2$-dimensional case.

\begin{proof}[Proof of Theorem \ref{th2} for $r=2$]
As in the proof of the $1$-dimensional case, we will construct our set $\{n_k\}$ as a union of three sets:
 $\{n_k\}=\{n_{k,0}\}\cup\{n_{k,1}\}\cup\{n_{k,2}\}$. These sets will be constructed by blocks, and they will depend from a parameter $H_N$, as in the $1$-dimensional case, but also from a parameter $L_N$ which will at each step $N$ be chosen very large:
$L_N \ll H_N \ll L_{N+1} \ll H_{N+1}$, where the sign $\ll$ means that the quantity on the right-hand side is much larger than the quantity on the left-hand side.
\par\smallskip
The main step in the proof is to obtain an analog of Lemma \ref{lem1}. In its statement, we will use the superscript ${ }^{(2)}$ so as to indicate that we are working with the $2$-dimensional approximation. This will simplify notation in the proof of the general multi-dimensional approximation.

\begin{lemma}\label{lem3}
Let $\varepsilon ^{(2)}$ be a  positive real number. There exist three integers $\Gamma^{(2)}$, $\Sigma  ^{(2)}$ and $\Theta ^{(2)}$  such that
if $L\ge 1$ is any integer, there exists an integer $Q^{(2)}\ge 1$  such that for any pair $(\lambda _{1},\lambda _{2})\in\T^{2}$ and any integer $H \ge 1$, 
 there exists an integer $n$ such that
\begin{eqnarray}\label{*}
|\lambda_1^n-1|<\varepsilon^{(2)}  \quad \textrm{and} \quad |\lambda_2^n-1|<\varepsilon ^{(2)}
\end{eqnarray}
and either 
\begin{eqnarray}\label{c1}
n\in \{Hq + 1 \textrm{ ; } 1\le q\le Q^{(2)} \}
\end{eqnarray}
or
\begin{eqnarray}\label{c2}
n\in\{H \Sigma^{(2)} L \}
\end{eqnarray}
or
\begin{eqnarray}\label{c3}
n\in\{H \Gamma ^{(2)} (L j + 1)\textrm{ ; } 1\le j\le \Theta^{(2)} \}.
\end{eqnarray}
\end{lemma}

\begin{proof}[Proof of Lemma \ref{lem3}]
Let $E_{\varepsilon^{(2)}}^{(2)}$ be the set of integers
\begin{eqnarray*}
E_{\varepsilon^{(2)}}^{(2)}=\left\{(a_1,a_2)\in\Z^{2}\setminus\{(0,0) \textrm{ ; } |a_1|+|a_2|< \dfrac{c_{2}}{\varepsilon ^{(2)}}\right\}\cdot
\end{eqnarray*}
 It is clear that if $(a_1,a_2)$ does not belong to $E_{\varepsilon^{(2)}}^{(2)}$, then for any $(\lambda_1,\lambda_2)\in\T^{2}$ and any choice of $Q^{(2)}$, condition (\ref{eq4}) in Corollary \ref{cor0} is automatically satisfied for $\varepsilon^{(2)}$.
\par\smallskip
Set $$\Gamma ^{(2)}=((\lfloor \frac{c_{2}}{\varepsilon ^{(2)}}\rfloor+1)!)^{2}.$$ This number has the property that for any element $(a_{1},a_{2})\in E_{\varepsilon^{(2)}}^{(2)}$, it is divisible by $a_{1}$, $a_{2}$ and $a_{1}a_{2}$, provided these numbers are non-zero.
Let $\varepsilon^{(1)}$ and $\delta ^{(2)}$ be two very small positive numbers, which will be chosen during the proof, depending from $\varepsilon^{(2)}$.
Now fix $(\lambda _{1},\lambda _{2})\in \T^{2}$. As in the proof of Lemma \ref{lem1}, we have several cases to consider, depending from whether $|\lambda _{1}^{Ha_{1}}\lambda _{2}^{Ha_{2}}-1|\le\delta ^{(2)}$ for some $(a_{1},a_{2})\in E_{\varepsilon^{(2)}}^{(2)}$ or not.

\par\smallskip
\textbf{Case 1:} there exists $(a_{1},a_{2})\in E_{\varepsilon^{(2)}}^{(2)}$ with $a_{1}a_{2}\not = 0$ such that $|\lambda _{1}^{Ha_{1}}\lambda _{2}^{Ha_{2}}-1|\le\delta ^{(2)}$.
\par\smallskip
Equivalently, replacing $a_{2}$ by $-a_{2}$, we assume that 
there exists $(a_{1},a_{2})\in E_{\varepsilon^{(2)}}^{(2)}$ with $a_{1}a_{2}\not = 0$ such that
$$|\lambda _{1}^{Ha_{1}}-\lambda _{2}^{Ha_{2}}|\le\delta ^{(2)}.$$ Lemma \ref{lem1}  applied to $\varepsilon^{(1)}$ gives us two integers $\Sigma  ^{(1)}$ and $\Theta ^{(1)}$  such that for all $\lambda \in\T$ and all $\tilde{H} \ge 1$, either 
$$|\lambda ^{\tilde{H}\Sigma  ^{(1)}L}-1|<\varepsilon ^{(1)}$$ or
$$|\lambda ^{\tilde{H}(Lj+1)}-1|<\varepsilon ^{(1)}\quad \textrm{ for some } j\in\{1,\ldots, \Theta ^{(1)}\}.$$ 
In particular, since $a_{1}a_{2}\not =0$ and $a_{1}a_{2}|\Gamma^{(2)}$, we can apply this to the integer 
$\tilde{H}$ defined by $\tilde{H}=H\frac{\Gamma ^{(2)}}{a_{1}a_{2}}$ and to $\lambda =\lambda_{1} ^{a_{1}}$: either 
$$|\lambda_{1} ^{{H}a_{1}\frac{\Gamma^{(2)}}{a_{1}a_{2}}\Sigma  ^{(1)}L}-1|<\varepsilon ^{(1)}$$
or
$$|\lambda_{1} ^{{H}a_{1}\frac{\Gamma^{(2)}}{a_{1}a_{2}}(Lj+1)}-1|<\varepsilon ^{(1)}\quad \textrm{ for some } j\in\{1,\ldots, \Theta^{(1)}\}.$$
\par\smallskip
\textbf{Case 1a:} we have $|\lambda_{1} ^{{H}\frac{\Gamma ^{(2)}}{a_{2}}\Sigma ^{(1)}L}-1|<\varepsilon^{(1)}.$
\par\smallskip
Then 
$$|\lambda_{1} ^{{H}{\Gamma ^{(2)}}\Sigma  ^{(1)}L}-1|<\varepsilon ^{(1)}|a_{2}|\le \varepsilon ^{(1)}\Gamma^{(2)}.$$
Moreover, since $|\lambda _{1}^{Ha_{1}}-\lambda _{2}^{Ha_{2}}|\le\delta ^{(2)}$, we have 
$$|\lambda_{2} ^{{H}a_{2}\frac{\Gamma ^{(2)}}{a_{1}a_{2}}\Sigma ^{(1)}L}-1|<\varepsilon^{(1)}+\delta^{(2)}\Gamma ^{(2)}\Sigma ^{(1)}L$$ and hence
$$|\lambda_{2} ^{{H}{\Gamma ^{(2)}}\Sigma  ^{(1)}L}-1|<(\varepsilon ^{(1)}+\delta ^{(2)}\Gamma ^{(2)}\Sigma ^{(1)}L){\Gamma ^{(2)}}.$$
Setting $\Sigma ^{(2)}={\Gamma ^{(2)}}\Sigma  ^{(1)}$ we get that
$$|\lambda _{1}^{H\Sigma  ^{(2)}L}-1|<\varepsilon ^{(2)}\quad \textrm{and}\quad 
|\lambda _{2}^{H\Sigma ^{(2)}L}-1|<\varepsilon^{(2)}$$ provided $\varepsilon ^{(1)}$ is chosen first, very small \wrt\ $\varepsilon ^{(2)}$ (but independent from $L$), and then $\delta^{(2)}$ is chosen very small \wrt\ $\varepsilon ^{(2)}$ and $L$. So (\ref{*}) and (\ref{c2}) are satisfied.
\par\smallskip
\textbf{Case 1b:} there exists a $j\in\{1,\ldots, \Theta^{(1)}\}$ such that
$$|\lambda_{1} ^{{H}\frac{\Gamma ^{(2)}}{a_{2}}(Lj+1)}-1|<\varepsilon^{(1)}.$$
Then 
\begin{eqnarray}\label{annexe}
 |\lambda_{1} ^{{H}{\Gamma ^{(2)}}(Lj+1)}-1|<\varepsilon ^{(1)}\Gamma ^{(2)}.
\end{eqnarray}
Using again that $|\lambda _{1}^{Ha_{1}}-\lambda _{2}^{Ha_{2}}|\le\delta^{(2)}$, we obtain that
\begin{eqnarray*}
 |\lambda_{1} ^{{H}a_{1}\frac{\Gamma ^{(2)}}{a_{1}a_{2}}(Lj+1)}-
 \lambda_{2} ^{{H}a_{2}\frac{\Gamma ^{(2)}}{a_{1}a_{2}}(Lj+1)}|<\delta  ^{(2)}\Gamma ^{(2)}
 (L\Theta ^{(1)}+1).
\end{eqnarray*}
Hence
\begin{eqnarray*}
 | \lambda_{2} ^{{H}\frac{\Gamma^{(2)}}{a_{1}}(Lj+1)}-1|<\varepsilon ^{(1)}+\delta  ^{(2)}\Gamma ^{(2)}
 (L\Theta^{(1)}+1), 
\end{eqnarray*}
and so 
\begin{eqnarray}\label{annexe2}
|\lambda _{}^{H{\Gamma^{(2)}}(Lj+1)}-1|<(\varepsilon ^{(1)}+\delta  ^{(2)}{\Gamma^{(2)}}
 (L\Theta^{(1)}+1)){\Gamma ^{(2)}}.
\end{eqnarray}
 Taking first $\varepsilon^{(1)}$ and then $\delta  ^{(2)}$ very small, we get from (\ref{annexe}) and (\ref{annexe2}) that
 $$|\lambda _{1}^{H{\Gamma ^{(2)}}(Lj+1)}-1|<\varepsilon ^{(2)}\quad \textrm{and}\quad 
 |\lambda _{2}^{H{\Gamma^{(2)}}(Lj+1)}-1|<\varepsilon ^{(2)}$$ for some $j\in\{1,\ldots, \Theta ^{(2)}\}$ where $\Theta ^{(2)}=\Theta ^{(1)}$. So (\ref{*}) and (\ref{c3}) are satisfied. Notice that $\Theta ^{(2)}$ does not depend from $L$.
\par\smallskip
Case $2$ is very similar to Case $1$: the assumptions are the same, except that we consider now the case where $a_{1}a_{2}=0$.
\par\smallskip
\textbf{Case 2:} there exists $(a_{1},a_{2})\in E_{\varepsilon^{(2)}}^{(2)}$ with $a_{1}a_{2} = 0$ such that $|\lambda _{1}^{Ha_{1}}\lambda _{2}^{H a_{2}}-1|\le\delta^{(2)}$.
\par\smallskip
For instance suppose that $a_{2}=0$ and $a_{1}\not =0$. Our assumption is then that
\begin{eqnarray}\label{annexe3}
 |\lambda _{1}^{H a_{1}}-1|\le\delta ^{(2)}.
\end{eqnarray}
We apply the dichotomy of the $1$-dimensional case to $\lambda _{2}$, with  $\varepsilon^{(1)}$ a very small positive number and $\tilde{H}=H\frac{\Gamma ^{(2)}}{a_{1}}$.
\par\smallskip
\textbf{Case 2a:} we have $|\lambda _{2}^{H\frac{\Gamma^{(2)}}{a_1}\Sigma^{(1)}L}-1|<\varepsilon ^{(1)}$.
\par\smallskip
Since $\Sigma^{(2)}={\Gamma^{(2)}}\Sigma^{(1)}$ we have 
$|\lambda _{2}^{H \Sigma^{(2)}L}-1|<\varepsilon ^{(1)}{\Gamma^{(2)}}$. Moreover since $a_1$ divides $\Gamma^{(2)}$, (\ref{annexe3}) implies that $|\lambda _{1}^{H\Gamma^{(2)}}-1|\le \delta^{(2)}\Gamma^{(2)}$ and hence
$$|\lambda _{1}^{H {\Gamma^{(2)}} \Sigma ^{(1)} L}-1|\le \delta^{(2)}{\Gamma^{(2)}}\Sigma^{(1)}L.$$ If $\varepsilon^{(1)}$
 and $\delta^{(2)}$ are sufficiently small,
$$|\lambda _{1}^{H\Sigma^{(2)} L}-1|<\varepsilon^{(2)}\quad \textrm{and}\quad
|\lambda _{2}^{H\Sigma^{(2)} L}-1|<\varepsilon^{(2)}$$
and (\ref{*}) and (\ref{c2}) are true. Again, $\Sigma ^{(2)}$ and $\Theta ^{(2)}$ do not depend from $L$.
\par\smallskip
\textbf{Case 2b:}  there exists a $j\in\{1,\ldots, \Theta ^{(1)}\}$ such that $|\lambda _{2}^{H\frac{\Gamma^{(2)}}{a_1} (Lj+1)}-1|<\varepsilon ^{(1)}$.
\par\smallskip
Then 
$$|\lambda _{2}^{H {\Gamma^{(2)}} (Lj+1)}-1|<\varepsilon ^{(1)}\Gamma^{(2)}.$$ Moreover from (\ref{annexe3}) we have
$$ |\lambda _{1}^{Ha_1 \frac{{\Gamma ^{(2)}}}{a_1} (Lj+1)}-1|<\delta^{(2)}{\Gamma^{(2)}}(L\Theta ^{(1)}+1)$$ and hence for
$\varepsilon ^{(1)}$ and $\delta^{(2)}$ small enough
$$|\lambda _{1}^{H \Gamma   ^{(2)}( Lj+1)}-1|<\varepsilon^{(2)}\quad \textrm{and}\quad
|\lambda _{2}^{H\Gamma   ^{(2)}( Lj+1)}-1|<\varepsilon^{(2)}$$
and (\ref{*}) and (\ref{c3}) are true.
\par\smallskip
At the end of these two cases, we see that it suffices to choose $\varepsilon ^{(1)}=\frac{\varepsilon ^{(2)}}{2\Gamma  ^{(2)}}$. Then the quantity $\delta^{(2)}$ is fixed small enough, depending from $\varepsilon ^{(2)}$ and $L$ but neither from $(\lambda _{1},\lambda _{2})$ nor from $H$, so that all the inequalities above are true. 
\par\smallskip
The last case, Case $3$, is the simplest one, where Corollary \ref{cor0} applies directly. In this last case we determine $Q^{(2)}$, which is the last quantity in the statement of Lemma \ref{lem3} not yet fixed.
\par\smallskip
\textbf{Case 3:} for each $(a_{1},a_{2})\in E_{\varepsilon ^{(2)}}^{(2)}$, $|\lambda _{1}^{H a_{1}}\lambda _{2}^{H a_{2}}-1|>\delta ^{(2)}$.
\par\smallskip
Let  $Q^{(2)}$ be an integer such that $$Q^{(2)}>\frac{c_2}{\delta ^{(2)}}\cdot$$
By
Corollary \ref{cor0}, there exists for all $(\mu _{1},\mu _{2})\in\T^{2}$ a $q$ with $1\le q\le Q^{(2)}$ such that $|\lambda _{1}^{H q}-\mu _{1}|<\varepsilon ^{(2)}$ and $|\lambda _{2}^{Hq}-\mu _{2}|<\varepsilon ^{(2)}$. Applying this with $\mu _{1}=\overline{\lambda} _{1}$ and $\mu _{2}=\overline{\lambda} _{2}$ gives
$$|\lambda _{1}^{Hq+1}-1|<\varepsilon ^{(2)} \textrm{ and } |\lambda _{2}^{Hq+1}-1|<\varepsilon ^{(2)},$$
i.e. that (\ref{*}) and (\ref{c1}) are satisfied.
\end{proof}

Lemma \ref{lem3} is proved. Let us summarize a bit more precisely for further use what we just proved:
\begin{corollary}\label{cor1}
With the notation of Lemma \ref{lem3}, there exists a positive number $\delta^{(2)}$
depending from $\varepsilon ^{(2)}$ and $L$ such that for all $(\lambda _{1},\lambda _{2})\in\T^{2}$ we have:
\par\medskip
-- if $|\lambda _{1}^{H a_{1}}\lambda _{2}^{Ha_{2}}-1|\le \delta^{(2)}$ for some $(a_{1},a_{2})\in E_{\varepsilon^{(2)}}^{(2)}$, then (\ref{*}) holds true for some integer
$n\in B_{\varepsilon^{(2)},1}^{(2)}\cup B_{\varepsilon^{(2)},2}^{(2)}$, where
$$B_{\varepsilon^{(2)},1}^{(2)}=\{H  \Sigma^{(2)} L\}
:=H C_{\varepsilon^{(2)},1}^{(2)} $$ and $$ B_{\varepsilon^{(2)},2}^{(2)}=\{H \Gamma  ^{(2)} (Lj + 1)\textrm{ ; } 1\le j\le \Theta^{(2)} \}:=HC_{\varepsilon^{(2)},2}^{(2)};$$
\par\medskip
-- if $|\lambda _{1}^{H a_{1}}\lambda _{2}^{H a_{2}}-1|>\delta ^{(2)}$ for each $(a_{1},a_{2})\in E_{\varepsilon^{(2)}}^{(2)}$, then (\ref{*}) holds true for some integer
$n\in B_{\varepsilon^{(2)},0}^{(2)}$, where 
$$ B_{\varepsilon^{(2)},0}^{(2)}=\{Hq+ 1 \textrm{ ; } 1\le q\le Q^{(2)}\}.$$
\end{corollary}

\begin{remark}
Observe that the sets 
$$C_{\varepsilon^{(2)},1}^{(2)}=\{ \Sigma^{(2)} L\}\quad \textrm{and}\quad C_{\varepsilon^{(2)},2}^{(2)}=\{ \Gamma ^{(2)} (L j + 1)\textrm{ ; } 1\le j\le \Theta^{(2)} \}$$ do not depend from $H$.
\end{remark}

\par\smallskip
As a corollary of Lemma \ref{lem3}, we obtain:

\begin{corollary}\label{cor11}
 Let $(\varepsilon _{N}^{(2)})$ be a sequence of positive numbers going to zero as $N$ goes to infinity. There exist three sequences of integers $(\Gamma  _{N}^{(2)})$, $(\Sigma  _{N}^{(2)})$ and $(\Theta _{N}^{(2)})$  such that
if $(L_{N})$ is any sequence of integers, there exists a sequence $(Q_{N}^{(2)})$ of integers such that for any sequence of integers $(H_{N})$, the union $\{n_{k}^{(2)}\}$ of the three sets
$$
\{n_{k,0}^{(2)}\}=\bigcup_{N\ge 1}B_{N,0}^{(2)}
\qquad
\{n_{k,1}^{(2)}\}=\bigcup_{N\ge 1}B_{N,1}^{(2)}
\qquad
\{n_{k,2}^{(2)}\}=\bigcup_{N\ge 1}B_{N,2}^{(2)}
$$
where 
\begin{eqnarray*}
B_{N,0}^{(2)}&=& B_{\varepsilon _{N}^{(2)},0}^{(2)}=\{H_{N}q+ 1 \textrm{ ; } 1\le q\le Q_{N}^{(2)}\}\\
B_{N,1}^{(2)}&=& B_{\varepsilon _{N}^{(2)},1}^{(2)}=\{H_N  \Sigma_N^{(2)} L_N\}\\
B_{N,2}^{(2)}&=& B_{\varepsilon _{N}^{(2)},2}^{(2)}= \{H_N \Gamma  _N^{(2)} (L_N j + 1)\textrm{ ; } 1\le j\le \Theta_N^{(2)} \}
\end{eqnarray*}
is a $2$-Bohr set. 
\end{corollary}

\par\smallskip
In order to finish the proof, it remains to show that if the two sequences $(L_{N})$ and $(H_{N})$ are well-chosen, there exist $\mu _{0},\mu _{1}, \mu _{2}\in\T$ such that for all $k$
$$|\mu _{0}^{n_{k,0}^{(2)}}-1|>\frac{1}{2}, \quad  |\mu _{1}^{n_{k,1}^{(2)}}-1|>\frac{1}{2}\quad  \textrm{and} \quad |\mu _{2}^{n_{k,2}^{(2)}}-1|>\frac{1}{2}\cdot$$ The construction of $\mu _{0}$ is done exactly as in the proof of the $1$-dimensional case, using the fact that $H_{N+1}$ can be chosen much larger than $H_{N}Q_{N}^{(2)}$. The construction of
$\mu _{1}$ is also the same:
whatever the choices of the integers $L_N$, we can get
$\mu_1\in\T$ with $|\mu_1^{H_{N}\Sigma_{N}^{(2)}L_N}-1|>\frac{1}{2}$ for all $N$ if the sequence $(H_N)$ grows sufficiently fast.
In order to construct $\mu _{2}$, we apply Lemma \ref{lem2} to the sequence $(m_{N})$ defined by $m_{2N}=H_{N}\Gamma   _{N}^{(2)}-1$ and $m_{2N+1}=H_{N}\Gamma _{N}^{(2)}L_{N}$: if we start from $H_{1}$ very large there exists $\mu _{2}\in\T$ very close to $-1$, such that for all $N$
$$|\mu _{2}^{H_{N}\Gamma_{N}^{(2)}-1}-1|\le M\dfrac{H_{N}\Gamma_{N}^{(2)}-1}{H_{N}\Gamma_{N}^{(2)}L_{N}}<M \dfrac{1}{L_{N}}<2^{-(N+1)}$$ if $L_{N}$ is sufficiently large,
and
$$|\mu _{2}^{H_{N}\Gamma_{N}^{(2)} L_{N}}-1|\le M\dfrac{H_{N}\Gamma   _{N}^{(2)}L_{N}}{H_{N+1}}\cdot$$ Then for all $j$ with $1\le j\le \Theta _{N}^{(2)}$, we have
$$|\mu _{2}^{H_{N}\Gamma _{N}^{(2)}L_{N}j}-1|\le M\dfrac{H_{N}\Gamma  _{N}^{(2)}L_{N}\Theta _{N}^{(2)}}{H_{N+1}}<2^{-(N+1)}$$ if $H_{N+1}$ is large enough. Hence $|\mu _{2}^{H_{N}\Gamma   _{N}^{(2)}(L_{N}j+ 1)}-\mu _{2}|\le 2^{-N}$ and
$$|\mu _{2}^{H_{N}\Gamma   _{N}^{(2)}(L_{N}j+ 1)}-1|\ge|\mu_2-1|-2^{-N}\ge\frac{1}{2}\cdot$$
So the set $\{n_{k}^{(2)}\}$ is non-recurrent  for the product of rotations $R_{\mu_{0}}\times R_{\mu_{1}}\times R_{\mu_{2}}$, and Theorem \ref{th2} is proved in the $2$-dimensional case.
\end{proof}

\section{The general multi-dimensional case}

Our aim now is to prove Theorem \ref{th2} in the general case by induction on $r\ge 3$. We are first going to prove the following analog of Lemma \ref{lem3} above, which will give explicitly the form of the sets $\{n_{k}^{(r)}\}$:

\begin{lemma}\label{th10}
Let $r\ge 3$ be an integer. For any $\varepsilon^{(r)}>0$ and any integer $L\ge 1$, there exist $2^{r-1}$ integers $\Delta  _{A}^{(r)}\ge 1$, where $A\subseteq \{1,\ldots, r-1\}$, and two integers $\Theta^{(r)}, Q^{(r)}$ such that the following holds true:
for any integer  $H \ge 1$, for any $r$-tuple $(\lambda _{1},\ldots, \lambda _{r})\in\T^{r}$, there exists an $n$ belonging to one of the sets
\begin{eqnarray}\label{aa}
B_{\varepsilon^{(r)},A}^{(r)}=\{H\Delta  _{A}^{(r)}(Lj+1) \textrm{ ; } 1\le j\le\Theta^{(r)}\}, \quad A\subseteq \{1,\ldots, r-1\}, A\not=\varnothing
\end{eqnarray}
\begin{eqnarray}\label{bb}
B_{\varepsilon^{(r)},\varnothing}^{(r)}=\{H\Delta  _{\varnothing}^{(r)}\}
\end{eqnarray}
and
\begin{eqnarray}\label{cc}
B_{\varepsilon^{(r)},0}^{(r)}=\{Hq+1 \textrm{ ; } 1\le q\le Q^{(r)}\}
\end{eqnarray}
such that 
$$\max_{i=1,\ldots, r}|\lambda _{i}^{n}-1|<\varepsilon^{(r)}.$$ 
We shall write $B_{\varepsilon^{(r)},A}^{(r)}$ as $B_{\varepsilon^{(r)},A}^{(r)}=HC_{\varepsilon^{(r)},A}^{(r)}$ where the sets $C_{\varepsilon^{(r)},A}^{(r)}$ do not depend from $H$.
 \end{lemma}

\begin{remark}\label{rem000}
The quantities $\Delta  _{A}^{(r)}$ and $\Theta^{(r)}$ (and $Q^{(r)}$ of course) depend from $\varepsilon^{(r)}$ and $L$ (but not from $H$). This is a difference with the $2$-dimensional case, where the quantities $\Gamma^{(2)}$, $\Sigma^{(2)}$ and $\Theta^{(2)}$ do not depend from $L$. Lemma \ref{th10} holds true for $r=2$ as well, with $\Delta  _{\varnothing}^{(2)}=\Sigma ^{(2)}L$ and $\Delta  _{\{1\}}^{(2)}=\Gamma  ^{(2)}$.
\end{remark}

\begin{proof}[Proof of Lemma \ref{th10}]
Fix $\varepsilon^{(r)}>0$ and  consider the set
$$E_{\varepsilon^{(r)}}^{(r)}=\{(a_{1},\ldots, a_{r})\in\Z^{r}\setminus\{(0,\ldots, 0)\} \textrm{ ; } \sum_{i=1}^{r}|a_{i}|<\frac{c_r}{{\varepsilon}^{(r)}}\},$$ where $c_r$ is the constant appearing in Corollary \ref{cor0}.
Here is the statement which we want to prove by induction on $r\ge 3$:
 
\begin{lemma}\label{prop0}
For any $\varepsilon^{(r)}>0$ and any integer $L\ge 1$, there exist $2^{r-1}$ integers $\Delta  _{A}^{(r)}$, $A\subseteq \{1,\ldots, r-1\}$, an integer
 $\Theta ^{(r)}\ge 1$ and a positive number $\delta ^{(r)}$ such that the following holds true: there exists an integer $Q^{(r)}\ge 1$  such that for any integer
 $H\ge 1$ and any $(\lambda _{1},\ldots, \lambda _{r})\in\T^{r}$, we have the following alternative:
 
 -- if there exists $(a_{1},\ldots, a_{r})\in E_{\varepsilon^{(r)}}^{(r)}$ such that $$|\lambda _{1}^{Ha_{1}}\lambda _{2}^{Ha_{2}}\ldots \lambda _{r}^{Ha_{r}}-1|\le\delta ^{(r)},$$ then there exists $n\in \bigcup_{A\subseteq\{1,\ldots, r-1\}}B_{\varepsilon^{(r)},A}^{(r)}$ such that 
 $$\max_{i=1,\ldots, r}|\lambda _{i}^{n}-1|<\varepsilon^{(r)};$$ 
 
  -- if for all 
 $(a_{1},\ldots, a_{r})\in E_{\varepsilon^{(r)}}^{(r)}$ we have $$|\lambda _{1}^{Ha_{1}}\lambda _{2}^{Ha_{2}}\ldots \lambda _{r}^{Ha_{r}}-1|>\delta ^{(r)},$$
 then there exists $n\in B_{\varepsilon^{(r)},0}^{(r)}$ such that 
 $$\max_{i=1,\ldots, r}|\lambda _{i}^{n}-1|<\varepsilon^{(r)}.$$

\end{lemma}

\begin{proof}[Proof of Lemma \ref{prop0}]
We prove Lemma \ref{prop0} by induction on $r\ge 2$. First, it follows from
Corollary \ref{cor1} that Lemma \ref{prop0} holds true for $r=2$, with $\Delta  _{\varnothing}^{(r)}=\Sigma^{(2)}L$ and $\Delta  _{\{1\}}^{(r)}=\Gamma^{(2)}$.
To carry out the induction step, 
let $r\ge 3$,  $\varepsilon ^{(r)}>0$ and $L\ge 1$.
Let $\varepsilon^{(r-1)}$ and $\varepsilon ^{(2)}$ be two positive numbers, and let $\delta^{(r-1)}$ and $\delta ^{(2)}$ be the two positive numbers associated to $\varepsilon^{(r-1)}$ and $L$, and $\varepsilon ^{(2)}$, respectively, given by Lemma \ref{prop0} for $r-1$ and by Lemma \ref{lem3}.
The numbers $\varepsilon^{(r-1)}$ and $\varepsilon ^{(2)}$ will be fixed during the proof, as well as the number $\delta^{(r)}>0$.
The quantity $\varepsilon ^{(r-1)}$ will be determined first, much smaller than $\varepsilon ^{(r)}$. This choice will determine $\delta ^{(r-1)}$. Then $\varepsilon ^{(2)}$ will be chosen much smaller than $\varepsilon^{(r-1)}$, $\delta^{(r-1)}$ and $\varepsilon^{(r)}$ (and this will determine $\delta ^{(2)}$), and only after this will the choice of $\delta^{(r)}$ be made, with $\delta ^{(r)}$ much smaller than any of the quantities considered before.
 \par\smallskip 
Fix $(\lambda _{1},\ldots, \lambda _{r})\in\T^{r}$. 
We consider again separately two cases, depending from whether $|\lambda _{1}^{Ha_{1}}\lambda _{2}^{Ha_{2}}\ldots \lambda _{r}^{Ha_{r}}-1| \le \delta ^{(r)}$ for some $(a_{1},\ldots, a_{r})\in E_{\varepsilon ^{(r)}}^{(r)}$ or not.
 \par\smallskip 
 \textbf{Case 1:} There exists $(a_{1},\ldots, a_{r})\in E_{\varepsilon ^{(r)}}^{(r)}$ such that 
 \begin{eqnarray}\label{id}
 |\lambda _{1}^{Ha_{1}}\lambda _{2}^{Ha_{2}}\ldots \lambda _{r}^{Ha_{r}}-1|\le\delta ^{(r)}.                                                                                                
\end{eqnarray}
 Without loss of generality we suppose that $(a_{2},\ldots, a_{r})\not = (0,\ldots, 0)$.
Then, replacing $a_1$ by $-a_1$, (\ref{id}) is equivalent to
\begin{eqnarray}\label{annexe4}
 |\lambda _{1}^{Ha_{1}}-\lambda _{2}^{Ha_{2}}\ldots \lambda _{r}^{Ha_{r}}|\le\delta ^{(r)}
\end{eqnarray}
for some $(a_{1},\ldots, a_{r})\in E_{\varepsilon ^{(r)}}^{(r)}$.
 Set $\nu _{1}=\lambda _{1}$ and $\nu _{2}=\lambda _{2}^{a_{2}}\ldots \lambda _{r}^{a_{r}}$. 
We have $|\nu _{1}^{Ha_{1}}-\nu _{2}^{H}|\le \delta ^{(r)}$, i.e. $$|\nu _{1}^{H\tilde{a}_{1}}-\nu _{2}^{H\tilde{a}_{2}}|\le \delta ^{(r)}$$ with 
$(\tilde{a}_{1},\tilde{a}_{2})=(a_{1},1)\in\Z^{2}\setminus\{(0,0)\}$. 
We have $|a_{1}|+1\le 2|a_{1}|<2\frac{c_{r}}{\varepsilon^{(r)}}<\frac{c_{2}}{\varepsilon^{(2)}}$ if ${\varepsilon^{(2)}}$
is small enough. So we get that $(\tilde{a}_{1},\tilde{a}_{2})$ belongs to $E_{\varepsilon ^{(2)}}^{(2)}$. 
If $\delta ^{(r)}$ is chosen so that $\delta ^{(r)}<\delta ^{(2)}$, 
 (\ref{annexe4}) implies that 
$$|\lambda _{1}^{Ha_{1}}-\lambda _{2}^{Ha_{2}}\ldots \lambda _{r}^{Ha_{r}}|\le\delta ^{(2)}.$$ We can now apply Corollary \ref{cor1} to $\varepsilon ^{(2)}$ and $L$: we get that there exists 
$n^{(2)}\in B_{\varepsilon ^{(2)},1}^{(2)}\cup
 B_{\varepsilon ^{(2)},2}^{(2)}$ such that
\begin{eqnarray}\label{zzz}
 |\nu _{1}^{n^{(2)}}-1|<\varepsilon^{(2)}\quad \textrm{ and }\quad 
 |\nu _{2}^{n^{(2)}}-1|<\varepsilon^{(2)}.
\end{eqnarray}
The integer
 $n^{(2)}$ is either equal to 
 $H\Sigma  ^{(2)}L$ or has the form $n^{(2)}=H\Gamma ^{(2)}(Lj+ 1)$ for some $1\le j\le \Theta ^{(2)}$.
In particular $n^{(2)}$ is a multiple of $H$, and we write $n^{(2)}=Hp^{(2)}$ with $p^{(2)}\in 
 C_{\varepsilon ^{(2)},1}^{(2)}\cup
 C_{\varepsilon ^{(2)},2}^{(2)}$. So we have
 $$|\lambda _{1}^{Hp^{(2)}}-1|<\varepsilon^{(2)}\quad \textrm{ and }\quad |\lambda _{2}^{Hp^{(2)}a_{2}}\ldots \lambda _{r}^{Hp^{(2)}a_{r}}-1|<\varepsilon^{(2)}.$$
 The $(r-1)$-tuple $(a_{2},\ldots, a_{r})$ satisfies
 $|a_{2}|+\ldots |a_{r}|<\frac{c_r}{\varepsilon^{(r)}}<\frac{c_{r-1}}{\varepsilon^{(r-1)}}$ if 
${\varepsilon^{(r-1)}}$ is chosen sufficiently small. 
Moreover $(a_{2}, \ldots, a_{r})\not = (0,\ldots, 0)$, and so 
 $(a_{2}, \ldots, a_{r})$ belongs to $ E_{\varepsilon^{(r-1)}}^{(r-1)}$.
If additionally $\varepsilon ^{(2)}$ is so small that ${\varepsilon^{(2)}}<{\delta^{(r-1)}}$, the induction assumption applied at rank $r-1$ to the $(r-1)$-tuple
$(\lambda _{2}, \ldots, \lambda _{r-1})$ and the integers $L$ and $\tilde{H}=Hp^{(2)}$ gives us 
an integer $n^{(r-1)}$ belonging to the set $\bigcup_{A'\subseteq\{1,\ldots, r-2\}}\tilde{H}
 C_{\varepsilon^{(r-1)},A'}^{(r-1)}$ such that
 $$\max_{i=2,\ldots, r}|\lambda _{i}^{n^{(r-1)}}-1|<{\varepsilon^{(r-1)}}.$$
 Notice that we can choose $\varepsilon ^{(r-1)}$, and so, by Corollary \ref{cor1}, determine $ C_{\varepsilon^{(r-1)},A'}^{(r-1)}$, before we fix $\varepsilon ^{(2)}$.
 Writing $n^{(r-1)}=\tilde{H}p^{(r-1)}$, we have $n^{(r-1)}={H}p^{(2)}
 p^{(r-1)}.$ Thus 
 $$\max_{i=2,\ldots, r}|\lambda _{i}^{{H}p^{(2)}
 p^{(r-1)}}-1|<{\varepsilon^{(r-1)}}.$$
 Moreover we have by (\ref{zzz}) that  $$|\lambda _{1}^{Hp^{(2)}}-1|<{\varepsilon^{(2)}},\quad \textrm{and so}\quad 
 |\lambda _{1}^{Hp^{(2)}p^{(r-1)}}-1|<{\varepsilon^{(2)}}p^{(r-1)}\cdot$$
 Now
 $$p^{(r-1)}\le\max_{A'\subseteq \{1,\ldots, r-2\}}\Delta  _{A'}^{(r-1)}(L\Theta^{(r-1)}+1).$$
Hence if we first fix ${\varepsilon^{(r-1)}}$ very small, we can then take ${\varepsilon^{(2)}}$ so small that (in addition to the other conditions) ${\varepsilon^{(2)}}p^{(r-1)}<{\varepsilon^{(r-1)}}$ for any $p^{(r-1)}\in \bigcup_{A'\subseteq\{1,\ldots, r-2\}}
 C_{\varepsilon^{(r-1)},A'}^{(r-1)}$.
We obtain then that 
 $$\max_{i=1,\ldots, r}|\lambda _{i}^{{H}p^{(2)}
 p^{(r-1)}}-1|<{\varepsilon^{(r-1)}},$$ and if we have taken at the beginning ${\varepsilon^{(r-1)}}<{\varepsilon^{(r)}}$ we get what we need, namely that
 $$\max_{i=1,\ldots, r}|\lambda _{i}^{{H}p^{(2)}
 p^{(r-1)}}-1|<{\varepsilon^{(r)}}.$$
 \par\smallskip
 It remains to check that the numbers $Hp^{(2)}p^{(r-1)}$ belong to a set of the form $B_{\varepsilon^{(r)},A}^{(r)}$. We know that such a number belongs to a set of the form
 $${H}C_{\varepsilon^{(2)}, \theta }^{(2)}.C_{\varepsilon^{(r-1)},A'}^{(r-1)}$$ for some $A'\subseteq \{1,\ldots, r-2\}$ and $\theta \in\{1,2\}$.
If $\theta =1$ such a set has the form
 \begin{eqnarray}\label{set1}
\{ H\Sigma ^{(2)} L \Delta _{A'}^{(r-1)} (Lj_{r-1}+1) \textrm{ ; }
1\le j_{r-1}\le \Theta ^{(r-1)} \}
 \end{eqnarray} if $A'\not=\varnothing $ and
 \begin{eqnarray}\label{set2}
\{ H\Sigma ^{(2)} L \Delta _{\varnothing}^{(r-1)} \}
 \end{eqnarray}
if $A'=\varnothing$.
If $\theta =2$ such a set has the form
\begin{eqnarray}\label{set3}
\{ H \Gamma^{(2)}\Delta _{A'}^{(r-1)}(Lj_2+1)(Lj_{r-1}+1) \textrm{ ; }
1\le j_{2}\le \Theta^{(2)}, \, 1\le j_{r-1}\le \Theta^{(r-1)}
\} 
\end{eqnarray}
if $A'\not=\varnothing $ and
 \begin{eqnarray}\label{set4}
\{ H\Gamma ^{(2)} \Delta _{\varnothing}^{(r-1)}(Lj_2+1) \textrm{ ; }
1\le j_{2}\le \Theta^{(2)}\} 
\end{eqnarray}
if $A'=\varnothing$.
Observing that the set in (\ref{set3}) is contained in 
 \begin{eqnarray*}
\{ H \Gamma^{(2)}\Delta _{A'}^{(r-1)}(Lj+1) \textrm{ ; }
1\le j\le \Theta^{(2)} \Theta ^{(r-1)}
\},
 \end{eqnarray*}
we see that these four sets have the required form: if we set 
$\Theta^{(r)}=\max(\Theta ^{(2)},\Theta ^{(r-1)})$, we have
$${H}C_{\varepsilon^{(2)},1}^{(2)}.C_{\varepsilon^{(r-1)},A'}^{(r-1)}\subseteq B_{\varepsilon^{(r)},A}^{(r)}$$
 with $A=A'\subset\{1,\ldots,r-1\}$ and $ \Delta _{A}^{(r)}=\Sigma ^{(2)}L\Delta _{A'}^{(r-1)}$ (observe that with this definition of $\Delta  _{A}^{(r)}$ we have $\Delta  _{A}^{(r)}=\Delta _{\varnothing}^{(2)}\Delta _{A'}^{(r-1)}$),
and
$${H}C_{\varepsilon^{(2)},2}^{(2)}.C_{\varepsilon^{(r-1)},A'}^{(r-1)}\subseteq B_{\varepsilon^{(r)},A}^{(r)}$$
with $A=A'\cup\{r-1\}\subseteq\{1,\ldots,r-1\} $ and $ \Delta  _{A}^{(r)}=\Gamma  ^{(2)}\Delta _{A'}^{(r-1)}$ (observe that in this case $\Delta  _{A}^{(r)}=\Delta _{\{1\}}^{(2)}\Delta _{A'}^{(r-1)}$).
 \par\smallskip
At the end of this case, the quantities $\Delta_{A} ^{(r)}$, $\Theta^{(r)}$  and $\delta^{(r)}$  are fixed. They depend from $\varepsilon^{(r)}$ and $L$, but not from $H$. It remains to determine $Q^{(r)}$.
 \par\smallskip
 \textbf{Case 2:} For all $(a_{1},\ldots, a_{r})\in E_{\varepsilon^{(r)}}^{(r)}$, $$|\lambda _{1}^{Ha_{1}}\lambda _{2}^{Ha_{2}}\ldots \lambda _{r}^{Ha_{r}}-1|>\delta ^{(r)}.$$
Let $Q^{(r)}$ be an integer such that $Q^{(r)}>\frac{c_r}{\delta ^{(r)}}$. By Corollary \ref{cor0}, there exists  an integer $q$ with $1\le q \le Q^{(r)}$ such that $\max_{i=1,\ldots, r}|\lambda _{i}^{Hq+1}-1|<\varepsilon^{(r)}$, and so 
 $\max_{i=1,\ldots, r}|\lambda _{i}^{n}-1|<\varepsilon^{(r)}$for some $n\in B_{\varepsilon^{(r)},0}^{(r)}$.
 \par\smallskip
We have thus proved Lemma \ref{prop0} at rank $r$, and the principle of induction completes the proof.
\end{proof}

A direct corollary of Lemma \ref{th10} is:

\begin{corollary}\label{cor111}
  Let $r\ge 3$, let $(\varepsilon _{N}^{(r)})$ be a sequence of positive numbers going to zero as $N$ goes to infinity and $(L_{N})$ be any sequence of integers. There exist $2^{r-1}$ sequences of integers $(\Delta  _{N,A}^{(r)})$, $A\subseteq\{1,\ldots, r-1\}$, and two sequences $(\Theta _{N}^{(r)})$ and $(Q_{N}^{(r)})$ of integers such that for any sequence of integers $(H_{N})$, the union $\{n_{k}^{(r)}\}$ of the sets
$$
\{n_{k,0}^{(r)}\}=\bigcup_{N\ge 1}B_{N,0}^{(r)}
\quad \textrm{ and }\quad 
\{n_{k,A}^{(r)}\}=\bigcup_{N\ge 1}B_{N,A}^{(r)}, \quad  A\subseteq\{1,\ldots, r-1\}
$$
where 
\begin{eqnarray*}
B_{N,0}^{(r)}&=& B_{\varepsilon _{N}^{(r)},0}^{(r)}=\{H_{N}q+ 1 \textrm{ ; } 1\le q\le Q_{N}^{(r)}\}\\
B_{N,\varnothing}^{(r)}&=& B_{\varepsilon _{N}^{(r)},\varnothing}^{(r)}=\{H_N  \Delta_{N,\varnothing}^{(r)}\}\\
B_{N,A}^{(r)}&=& B_{\varepsilon _{N}^{(r)},A}^{(r)}= \{H_N\Delta    _{N,A}^{(r)} (L_N j + 1)\textrm{ ; } 1\le j\le \Theta_N^{(r)} \}
\end{eqnarray*}
is an $r$-Bohr set. 
\end{corollary}

The quantities $\Delta  _{N,A}^{(r)}$, $\Theta _{N}^{(r)}$ and $Q_{N}^{(r)}$ are obtained by applying Lemma \ref{th10} to the numbers $\varepsilon _{N}^{(r)}$ and $L_{N}$.
In order to finish the proof of Theorem \ref{th2}, it remains to prove that
 $\{n_{k}^{(r)}\}$ is not a recurrence set for some product of $2^{r-1}+1$ rotations if the sequences $(L_{N})$ and $(H_{N})$ grow sufficiently fast.

\begin{proposition}\label{prop1}
Let $r\ge 3$. If the sequences $(L_{N})$ and $(H_{N})$ grow sufficiently fast, with
$L_{N}\ll H_{N}\ll L_{N+1}$, then there exist $2^{r-1}+1$ elements $\mu _{0}$ and $ \mu _{A}$ of $\T$ such that for any $k$ and any $A\subseteq \{1,\ldots, r-1\}$,
 \begin{eqnarray}\label{ee1}
  |\mu _{0}^{n_{k,0}^{(r)}}-1|>\frac{1}{2}
 \quad \textrm{and}\quad 
 |\mu _{A}^{n_{k,A}^{(r)}}-1|>\frac{1}{2}\cdot
\end{eqnarray}
\end{proposition}

\begin{proof}[Proof of Proposition \ref{prop1}]
We obtain  $\mu _{0}$ in exactly the same way as in the proof of Theorem \ref{th2}. The construction of $\mu _{\varnothing}$ is also similar: whatever the choice of $L_{N}$, we can ensure that for some $\mu _{\varnothing}\in\T$ with $|\mu _{\varnothing}+1|<1$,
$$|\mu _{\varnothing}^{H_{N}\Delta  _{N,\varnothing}^{(r)}-1}-1|<2^{-N}$$ for all $N\ge 1$, provided the sequence $(H_{N})$ grows sufficiently fast. This shows that $$|\mu _{\varnothing}^{H_{N}\Delta  _{N,\varnothing}^{(r)}}-1|\ge \frac{1}{2}$$ for all $N\ge 1$. 
Let now $A\subseteq\{1,\ldots, r-1\}$, $A\not =\varnothing$,  and consider the sequence $(m_{N})$ defined by
$m_{2N}=H_{N}\Delta _{N,A}^{(r)}-1$ and $m_{2N+1}=H_{N}\Delta _{N,A}^{(r)}L_{N}$.
The argument is again the same as in the proof of the $2$-dimensional case (the fact that $\Delta  _{N,A}^{(r)}$ depends from $L_{N}$ does not play a role here).
By Lemma \ref{lem3} again, if $H_{1}$ is very large we can find $\mu _{A}$ very close to $-1$ such that for all $N\ge 1$,
$$|\mu _{A}^{H_{N}\Delta _{N,A}^{(r)}-1}-1|\le M\frac{H_{N}\Delta _{N,A}^{(r)}-1}{H_{N}\Delta _{N,A}^{(r)}L_{N}}
\quad \textrm{and}\quad 
|\mu _{A}^{H_{N}\Delta _{N,A}^{(r)}L_{N}}-1|\le M\frac{H_{N}\Delta _{N,A}^{(r)}L_{N}}{H_{N+1}}\cdot$$
Thus $$|\mu _{A}^{H_{N}\Delta _{N,A}^{(r)}L_{N}j}-1|\le M\Theta_{N}^{(r)}\frac{H_{N}\Delta _{N,A}^{(r)}L_{N}}{H_{N+1}}$$
for all $1\le j\le \Theta _{N}^{(r)}$. It follows that for all $N\ge 1$ and all $1\le j\le \Theta _{N}^{(r)}$,
 $$|\mu _{A}^{H_{N}\Delta _{N,A}^{(r)}(L_{N}j+1)}-\mu_{A}|\le \frac{M}{L_{N}} + M\Theta_{N}^{(r)}\frac{H_{N}\Delta _{N,A}^{(r)}L_{N}}{H_{N+1}}\cdot$$ If for each $N\ge 1$ we take $L_{N}$ very large, and then choose $H_{N+1}$ very large \wrt\ $L_{N}$ and $H_{N}$, we can ensure for instance that 
$$|\mu _{A}^{H_{N}\Delta _{N,A}^{(r)}(L_{N}j+1)}-\mu_{A}|\le 2^{-N}, \textrm{ so that } 
|\mu _{A}^{H_{N}\Delta _{N,A}^{(r)}(L_{N}j+1)}-1|\ge |\mu _{A}-1|-2^{-N}>\frac{1}{2}$$ if $|\mu _{A}-1|>1$. 
This proves that $|\mu _{A}^{n_{k,A}^{(r)}}-1|>\frac{1}{2}$ for all $k$, and
 Proposition \ref{prop1} is proved.
\end{proof}

We have thus exhibited a product of $2^{r-1}+1$  rotations for which $\{n_{k}^{(r)}\}$ is not a recurrence set. Theorem \ref{th2} is proved.
 \end{proof}
 
 \begin{remark}\label{remrem}
  Inspection of the proof of Theorem \ref{th2} shows that the same phenomenon appears for general $r$ as for $r=1$: the sets $B_{N,A}^{(r)}$, $B_{N,0}^{(r)}$ are by construction intertwined, and they cannot be forced far away one from another. Indeed, for any $\varepsilon >0$, let us write $\Gamma ^{(2)}$, $\Sigma  ^{(2)}$ and $\Theta ^{(2)}$ as
  $\Gamma ^{(2)}_{\varepsilon }$, $\Sigma  ^{(2)}_{\varepsilon }$ and $\Theta ^{(2)}_{\varepsilon }$ in order to indicate their dependance from $\varepsilon $. It follows from the proofs of Lemma \ref{lem1} and \ref{lem3} that $\Gamma ^{(2)}_{\varepsilon }$ divides $\Sigma  ^{(2)}_{\varepsilon }$, that $\Theta  ^{(2)}_{\varepsilon }$ is much larger than $\Sigma  ^{(2)}_{\varepsilon }$, and that if $\varepsilon '$ is much smaller than $\varepsilon $, $\Sigma  ^{(2)}_{\varepsilon }$ divides $\Sigma  ^{(2)}_{\varepsilon' }$. Also, the proof of Lemma \ref{lem1} yields that $\Sigma  ^{(2)}_{\varepsilon }=\Gamma ^{(2)}_{\varepsilon }\Sigma  ^{(1)}_{\varepsilon^{(1)} }$, where $\varepsilon^{(1)}$ is much smaller than $\varepsilon$. So if $M\ge 1$ is an integer, and if we take $\varepsilon^{(1)}$ small enough, it follows from Remark \ref{rem00} that we can ensure that $\Sigma  ^{(1)}_{\varepsilon^{(1)} }$ is divisible with $M$.
Looking more closely at the expressions of $\Delta _{A}^{(r)}$ in the proof of Lemma \ref{prop0}, we see that 
$$\Delta  _{\varepsilon ^{(r)},A'}^{(r)}=\Sigma  _{\varepsilon ^{(2)}}^{(2)}L\Delta  _{\varepsilon ^{(r-1)},A'}^{(r-1)} 
=\Gamma  _{\varepsilon ^{(2)}}^{(2)}\Sigma  ^{(1)}_{\varepsilon^{(1)} }L\Delta  _{\varepsilon ^{(r-1)},A'}^{(r-1)}
\;\textrm{ and }\;
 \Delta  _{\varepsilon ^{(r)},A'\cup\{r-1\}}^{(r)}=\Gamma _{\varepsilon ^{(2)}}^{(2)}\Delta  _{\varepsilon ^{(r-1)},A'}^{(r-1)}$$
  for $A'\subseteq \{1,\ldots, r-2\}$, where $\varepsilon ^{(2)}$ is extremely small. Now $\varepsilon ^{(r-1)}$ is small, but only compared to $\varepsilon ^{(r)}$, and if we take $\varepsilon ^{(2)}$ small enough we can ensure that 
  $\Sigma _{\varepsilon ^{(2)}}^{(2)}$ is divisible with any of the numbers $ \Gamma ^{(2)}_{\varepsilon^{(2)} }\Delta  _{\varepsilon ^{(r-1)},A'}^{(r-1)}$, where $A'$ runs over all subsets of $\{1,\ldots, r-2\}$. It follows from this observation that given two distinct subsets $A_{1}$ and $A_{2}$ of $\{1,\ldots, r-1\}$, one of the two integers
  $\Delta  _{\varepsilon ^{(r)},A_{1}}^{(r)}$ and   $\Delta  _{\varepsilon ^{(r)},A_{2}}^{(r)}$ is always divisible with  the other. As   $\Theta   _{\varepsilon ^{(r)}}^{(r)}$ is very large compared to all the numbers 
   $\Delta  _{\varepsilon ^{(r)},A}^{(r)}$, the two arithmetic progressions $\{H  \Delta  _{\varepsilon ^{(r)},A_{1}}^{(r)}(Lj+1) \textrm{ ; }1\le j\le   \Theta   _{\varepsilon ^{(r)}}^{(r)}\}$ and $\{H  \Delta  _{\varepsilon ^{(r)},A_{2}}^{(r)}(Lj+1) \textrm{ ; }1\le j\le   \Theta   _{\varepsilon ^{(r)}}^{(r)}\}$ are necessarily intertwined. Lastly, since $Q^{(r)}$ is much larger than any integer $\Delta  _{\varepsilon ^{(r)},A}^{(r)}L
\Theta  _{\varepsilon ^{(r)}}^{(r)}$, these arithmetic progressions are also intertwined with the set $B  _{\varepsilon ^{(r)},0}^{(r)}$.
 \end{remark}

\section{Final remarks}
\subsection{Back to Question \ref{q2}}
Let $(N_{r})_{r\ge 1}$ be an increasing sequence of integers, and $(\varepsilon_{r})_{r\ge 1}$ a sequence of positive real numbers going to $0$ as $r$ goes to infinity. Consider the set $\{n_{k}\}$ defined by $\{n_{k}\}=\bigcup_{r\ge 1}\{n_{k}^{(r)}\textrm{ ; } k\in I_{N_{r}}\}$ where $\{n_{k}^{(r)}\textrm{ ; } k\in I_{N_{r}}\}$ is the part of the set $\{n_{k}^{(r)}\}$ constructed at step $N_{r}$, with suitable integers $L_{N_{r}}$ and $H_{N_{r}}$:
\begin{eqnarray*}
 \{n_{k}^{(r)}\textrm{ ; } k\in I_{N_{r}}\}&=&\{H_{N_{r}}q+ 1 \textrm{ ; }1\le q\le Q_{N_{r}}^{(r)}\}\cup\{H_{N_{r}}\Delta  _{N_{r},\varnothing}^{(r)}\}\\
 &\cup&\bigcup_{A\subseteq\{1,\ldots, r-1\}, A\not =\varnothing}\{H_{N_{r}}\Delta   _{N_{r},A}^{(r)}
  (L_{N_{r}}j+ 1)\textrm{ ; }
 1\le j\le \Theta _{N_{r}}^{(r)}\}.
\end{eqnarray*}
All the sets $\{n_{k}^{(r)}\textrm{ ; } k\in I_{N_{r}}\}$ are disjoint, and very far away one from another. For all $r\ge 1$ and $(\lambda _{1}, \lambda _{2},\ldots,\lambda _{r})\in \T^{r}$, we can consider for $s\ge r$ the $s$-tuple
$(\lambda _{1},\ldots, \lambda _{r},1,\ldots, 1)\in\T^{s}$. We know from Theorem \ref{th10} that there exists a $k\in I_{N_{s}}$ such that 
$$\max_{i=1,\ldots, r}|\lambda_{i} ^{n_{k}^{(s)}}-1|<\varepsilon_{s},$$ and so we see that 
the set $\{n_{k}\}$ is a Bohr set. Moreover it is not difficult to see from the construction that if $(\Pi_{N_{r}}^{(r)})_{r\ge 1}$ is any sequence of integers, the sets
\begin{eqnarray*}
 \{n_{k}^{(r)}\textrm{ ; } k\in I_{N_{r}}\}&=&\{H_{N_{r}}q+ 1 \textrm{ ; }1\le q\le Q_{N_{r}}^{(r)}\}\cup\{H_{N_{r}}\Delta  _{N_{r},\varnothing}^{(r)}\}\\
 &\cup&\bigcup_{A\subseteq\{1,\ldots, r-1\}, A\not =\varnothing}\{H_{N_{r}}\Delta   _{N_{r},A}^{(r)}
  (L_{N_{r}}j+ 1)\textrm{ ; }
 \Pi_{N_{r}}^{(r)}\le j\le \Theta _{N_{r}}^{(r)}\}.
\end{eqnarray*}
are also $r$-Bohr provided $\Theta _{N_{r}}^{(r)}$ is sufficiently large for each $r\ge 1$. Hence $\{n_{k}^{(r)}\textrm{ ; } k\in I_{N_{r}}\}$ is a Bohr set as well in this case.
\par\smallskip
All these sets $\{n_{k}\}$ are ``small'' (in particular they have density zero), and, more importantly, they have a very explicit arithmetical structure. We do not know whether $\{n_{k}\}$ can be non-recurrent for some suitable choice of the parameters in the construction, but we do know that, for some particular choices, $\{n_{k}\}$
is a recurrence set, and even a Poincar\'{e} set. This leaves the following question open to further investigation:

\begin{question}
Is it possible to choose the parameters in the construction of the set $\{n_{k}\}$ above in such a way that $\{n_{k}\}$ is not a recurrence set? 
\end{question}

\subsection{Other classes of non-recurrent systems for $\{n_{k}^{(r)}\}$}
We have seen that each one of the sets $\{n_{k}^{(r)}\}$ constructed in the proof of Theorem \ref{th2} is not recurrent for some product of $2^{r-1}+1$ rotations. These are very specific dynamical systems, and one can wonder whether there are other ``natural'' dynamical systems which would be non-recurrent \wrt\ $\{n_{k}^{(r)}\}$. In particular in a recent work \cite{BDLR} Bergelson, Del Junco, Lema\'nczyk and Rosenblatt initiated the study of sets which are non-recurrent in the measure-theoretic sense for weakly mixing dynamical systems. Thus the question naturally arises: 
is it possible to find an $r$-Bohr set which would be non-recurrent (in the measure-theoretic sense) for some  weakly mixing dynamical system? Such $r$-Bohr sets would necessarily have density zero, so that the examples of \cite{K} cannot have this property. It is possible to show that for each $r\ge 1$, each one of the sets 
 $\{n_{k}^{(r)}\}$ obtained in Theorem \ref{th2} is non-recurrent for some  weakly mixing dynamical system. This will be developed in \cite{Gr}. 
%

\end{document}